\documentclass[10pt,a4paper]{amsart}

\usepackage[utf8]{inputenc}
\usepackage[english]{babel}

\usepackage{paralist}

\usepackage[bookmarks,english,bookmarksopen,unicode]{hyperref}
\hypersetup{pdftitle={Integral Brauer-Manin obstructions for sums of two squares and a power},pdfauthor={Fabian Gundlach}}

\usepackage{amssymb}
\usepackage{amsmath}
\usepackage{amsthm}

\title[Brauer--Manin obstructions for sums of two squares and a power]{Integral Brauer--Manin obstructions for sums of two squares and a power}
\author{Fabian Gundlach}
\address{Mathematisches Institut, Ludwig-Maximilians-Universität München, Theresienstr. 39, 80333 München, Germany}
\email{fabian.gundlach@campus.lmu.de}

\usepackage[english,noabbrev,sort]{cleveref}

\usepackage{algorithmicx}
\usepackage{algpseudocode}

\newtheorem{theorem}{Theorem}
\newtheorem{corollary}[theorem]{Corollary}
\newtheorem{lemma}[theorem]{Lemma}

\newtheorem*{hypothesis}{Hypothesis}
\numberwithin{theorem}{section}
\newtheorem{mtheorem}{Theorem}

\theoremstyle{definition}
\newtheorem{definition}[theorem]{Definition}

\newtheorem{remark}[theorem]{Remark}

\newcommand{\Z}{\mathbb{Z}}
\newcommand{\Q}{\mathbb{Q}}
\newcommand{\R}{\mathbb{R}}
\newcommand{\A}{\mathbb{A}}
\newcommand{\X}{\mathfrak{X}}

\renewcommand{\O}{\mathcal{O}}

\newcommand{\leg}[2]{\left(\frac{#1}{#2}\right)}

\DeclareMathOperator{\Br}{Br}
\DeclareMathOperator{\Gal}{Gal}

\DeclareMathOperator{\inv}{inv}
\DeclareMathOperator{\Spec}{Spec}

\usepackage{perpage}
\MakePerPage{footnote}

\makeatletter
\let\OldStatex\Statex
\renewcommand{\Statex}[1][3]{%
  \setlength\@tempdima{\algorithmicindent}%
  \OldStatex\hskip\dimexpr#1\@tempdima\relax}
\makeatother


\begin{document}

\maketitle

\begin{abstract}
We use Brauer--Manin obstructions to explain failures of the integral Hasse principle and strong approximation away from $\infty$ for the equation $x^2+y^2+z^k=m$ with fixed integers $k\geq3$ and $m$. Under Schinzel's hypothesis~(H), we prove that Brauer--Manin obstructions corresponding to specific Azumaya algebras explain all failures of strong approximation away from $\infty$ at the variable $z$. Finally, we present an algorithm that, again under Schinzel's hypothesis (H), finds out whether the equation has any integral solutions.
\end{abstract}

\setcounter{tocdepth}{1}
\tableofcontents
\newcommand{\sectionbreak}{\clearpage\phantomsection}
\newcommand{\subsectionbreak}{\clearpage\phantomsection}

\section{Introduction}

For integers $k\geq2$ and $m$ we consider the equation
\begin{equation}\label{zkeq}
x^2+y^2+z^k=m.
\end{equation}
For $k=2$ the famous theorem of Gauß about sums of three squares says that (\ref{zkeq}) has an integral solution if and only if $m\geq0$ and $m$ is not of the form $4^u(8l+7)$ for non-negative integers $u$ and $l$. The non-existence of integral solutions can in this case always be explained by the non-existence of real or 2-adic solutions.

Vaughan conjectured in \cite[Chapter 8]{vaughan} that for sufficiently large $m$ there is an integral solution to (\ref{zkeq}) satisfying $x,y,z\geq0$ whenever for each prime $q$ there is some solution $(x,y,z)\in\Z_q$ to the above equation such that $q\nmid\gcd(x,y,z)$. For odd $k\geq3$ such local solutions always exist. His conjecture would then imply the integral Hasse principle for sufficiently large $m$.

This was however disproved by Jagy and Kaplanski in \cite{jagykaplansky}. They gave an elementary proof using quadratic reciprocity that there is no integral solution if $k=9$ and $m=(6p)^3$ for some prime $p\equiv1\mod4$. The remark following their theorem mentions that if $k$ is an odd composite integer, then for infinitely many $m\in\Z$ \cref{zkeq} has no solution.

Dietmann and Elsholtz gave examples of failures of strong approximation in \cite{dietmannpub} for $k=4$ and more general ones in \cite{dietmann} for arbitrary $k\geq2$.

Brauer--Manin obstructions were originally introduced by Manin to explain failures of the Hasse principle for rational points on certain cubic surfaces (see for example \cite{manin}). For an overview of further developments of Brauer--Manin obstructions for the Hasse principle and weak approximation for rational points see \cite{peyre}.

This method was adapted to integral points and applied to quadratic forms such as $x^2+y^2+z^2=m$ by Colliot-Thél\`ene and Xu in \cite{ctxu1}. Further examples of failures of the integral Hasse principle and strong approximation explained by Brauer--Manin obstructions are given in \cite{kreschtschinkel}, \cite{ct7}, \cite{ct8} and \cite{ct9}.

We show that the counterexample to the integral Hasse principle given in \cite{jagykaplansky} can be explained by a Brauer--Manin obstruction (see \Cref{jagythm}).

Furthermore, we systematically find new counterexamples to the integral Hasse principle and strong approximation:

\begin{mtheorem}\label{nweakgeneralm}
The following equations do not fulfill strong approximation away from $\infty$ due to Brauer--Manin obstructions:
\begin{align*}
x^2+y^2+z^k=n^k,&\qquad\text{$k\geq3$ odd, $n\equiv1\mod4$}\\
x^2+y^2+z^k=n^k,&\qquad\text{$k\geq2$ even, $n>0$}
\end{align*}
\end{mtheorem}
\begin{proof}
See \Cref{nweakgeneral}.
\end{proof}

Our second goal is to show the fulfillment of the integral Hasse principle and strong approximation away from $\infty$ at the variable $z$ in case there is no Brauer--Manin obstruction via certain elements of the Brauer group. Unfortunately, we can only do this under assumption of Schinzel's hypothesis (H), a generalization of Dirichlet's theorem on primes within arithmetic progressions to prime values of polynomials.

Schinzel's hypothesis (H) has been employed by Colliot-Th\'el\`ene and Sansuc in \cite{ct2} to prove the Hasse principle and weak approximation for rational solutions of equations similar to (\ref{zkeq}). This technique has subsequently been used for example in \cite{ct3}, \cite{ct6}, \cite{ct4}, \cite{witt} and \cite{wei}.

However, as far as we know, for integral points the potential use of Schinzel's hypothesis (H) was so far only briefly mentioned in Remark (v) on pages 618--619 of \cite{ct6}.

\begin{mtheorem}\label{schinzelthmm}
Let $k\geq3$ be an odd integer. Under Schinzel's hypothesis (H) each solution $(x_v,y_v,z_v)_v\in\prod_v\Z_v^3$ to \cref{zkeq} without any Brauer--Manin obstruction generated by Azumaya algebras of the form described in \Cref{azuconst} can be approximated with respect to the variable $z$ by integral solutions to \cref{zkeq}.
\end{mtheorem}
\begin{proof}
See \Cref{schinzelthm}.
\end{proof}

Jagy and Kaplanski conjectured in \cite{jagykaplansky} that (\ref{zkeq}) has an integral solution whenever $k$ is an odd prime.
\begin{mtheorem}\label{primekm}
Let $k$ be an odd prime. Under Schinzel's hypothesis (H) every integer is of the form $x^2+y^2+z^k$ for integral $x,y,z\in\Z$.
\end{mtheorem}
\begin{proof}
See \Cref{primek}.
\end{proof}

For each prime $p$ and $x\in\Q_p^\times$ let $r_p(x):=\frac x{p^{v_p(x)}}$.

\begin{mtheorem}\label{abthmm}
Let $k$ be the product of two primes $a,b\equiv1\mod4$ and let $m\in\Z\setminus\{0\}$.

For the existence of integral solutions to \cref{zkeq}, it is necessary and under Schinzel's hypothesis (H) also sufficient that the following two statements are both true.
\begin{itemize}
\item There is no $n\equiv6\mod8$ such that $m=n^a$ and for each prime $p\equiv3\mod4$ dividing $n$:

$b\nmid v_p(n)$ or $2\mid v_p(n)$ or there is no $z'\in\{0,\dots,p-1\}$ such that
\[p\mid r_p(n)^{a-1}+\dots+z'^{(a-1)b}.
\]
\item There is no $n\equiv6\mod8$ such that $m=n^b$ and for each prime $p\equiv3\mod4$ dividing $n$:

$a\nmid v_p(n)$ or $2\mid v_p(n)$ or there is no $z'\in\{0,\dots,p-1\}$ such that
\[p\mid r_p(n)^{b-1}+\dots+z'^{(b-1)a}.
\]
\end{itemize}
\end{mtheorem}
\begin{proof}
See \Cref{abthm}.
\end{proof}

For $m\in\Z$ and odd $k\geq1$ an algorithm is given in \Cref{algosection}, which, using Schinzel's hypothesis (H), determines whether $m$ is of the form $x^2+y^2+z^k$.

Finally, lists of small positive integers not of the form $x^2+y^2+z^k$ are given for small odd $k$.

\textbf{Acknowledgements.} We thank Jean-Louis Colliot-Thélène, Christian Elsholtz, Dasheng Wei and the referee for their comments.

\section{Preliminaries}

From now on, let $K$ be a number field, $\Omega$ the set of places of $K$ and $\Omega_\infty\subseteq\Omega$ the set of archimedian places of $K$. Let $K_v$ be the completion of $K$ with respect to $v$ for each $v\in\Omega$. Let $\O_v$ be the corresponding valuation ring for each $v\in\Omega\setminus\Omega_\infty$ and let $\O_v:=K_v$ for each $v\in\Omega_\infty$. The valuation associated to $v\in\Omega\setminus\Omega_\infty$ is called $v_v$. The ring $\O:=\{x\in K\mid v_v(x)\geq0\ \forall v\in\Omega\setminus\Omega_\infty\}$ is called the ring of integers of $K$.

In this section, let $X$ be a variety over $K$.

For topological rings $R$ over $K$, the set of $R$-rational points $X(R)$ obtains the induced topology.

Given a class of varieties, one often wants to know whether the existence of local solutions implies the existence of global solutions, or, even better, whether the existence of local integral solutions implies the existence of integral solutions. This leads to

\begin{definition}\label{adeledef}
Let $S$ be a subset of $\Omega$.
The set of $S$-adeles
\[
\A_S:=\Big\{(x_v)_{v\in\Omega\setminus S}\in\prod_{v\in\Omega\setminus S}K_v \;\Big\vert\; x_v\in\O_v\text{ for almost every }v\in\Omega\setminus S\Big\}
\]
is a ring by coordinatewise addition and multiplication.
The ring $\A:=\A_\emptyset$ is called the \emph{adele ring of $K$}.
The sets
\[
\Big\{\prod_{v\in\Omega\setminus S}A_v\;\Big\vert\; A_v=\O_v\text{ for almost all }v\in\Omega\setminus S\text{ and $A_v$ open in $K_v$ for all }v\in\Omega\setminus S\Big\}
\]
define a basis for the topology on $\A_S$.

For $S\subsetneq\Omega$ the field $K$ may be diagonally embedded into $\A_S$ as for every $x\in K$ there are only finitely many $v$ such that $x\not\in\O_v$. Below, the images of these embeddings are identified with $K$.
\end{definition}

Given a variety $X$ and some $S\subsetneq\Omega$, obviously $X(K)\subseteq X(\A_S)$. It is of interest how $X(K)$ relates to $X(\A_S)$.

\begin{definition}
We say that the variety $X$ satisfies \emph{strong approximation away from $S\subset\Omega$} if $\overline{X(K)}=X(\A_S)$ (where the closure is taken inside $X(\A_S)$), i.e., if $X(K)$ is dense in $X(\A_S)$.
\end{definition}

An introduction to Brauer--Manin obstructions can be found in \cite{skorobogatov}.

\begin{definition}[{\cite[Chapter IV]{milneetale}}]
An $\O_X$-algebra $A$ is called an \emph{Azumaya algebra over $X$} if it is coherent (i.e., there is some open covering by affine schemes $U_i\cong\Spec A_i$, such that $A|_{U_i}\cong \widetilde M_i$ for some finitely generated $A_i$-module $M_i$ for each $i$) and if $A_x\otimes_{\O_{X,x}}\kappa(x)$ is a central simple algebra over the residue field $\kappa(x)$ for every $x\in X$.

If furthermore $k$ is a field extension of $K$, then for each $x\in X(k)$ (i.e., each morphism $x: \Spec k\rightarrow X$ of $K$-schemes) let $A(x):=A_{x(\eta)}\otimes_{\O_{X,x(\eta)}}k$.
\end{definition}

\begin{remark}
If $A$ is an Azumaya algebra over $X$, then $A(x)$ is a central simple $k$\nobreakdash-algebra for each $x\in X(k)$, as $A_{x(\eta)}\otimes_{\O_{X,x(\eta)}}\kappa(x(\eta))$ is a central simple $\kappa(x(\eta))$-algebra and $k$ is a $\kappa(x(\eta))$-algebra by the morphism $x$, such that
\[A(x)\cong \left(A_{x(\eta)}\otimes_{\O_{X,x(\eta)}}\kappa(x(\eta))\right)\otimes_{\kappa(x(\eta))}k.
\]
\end{remark}

\begin{definition}
For $v\in\Omega$ let $\inv_v:\Br(K_v)\rightarrow\Q/\Z$ be the invariant map from local class field theory. For simplicity, we will refer to the class of $a\in\Q$ in $\Q/\Z$ by $a$, too.
\end{definition}

\begin{theorem}[Brauer--Manin obstruction, {\cite[Chapter 5.2]{skorobogatov}} and {\cite[Section 2]{ct8}}]\label{skorothm}
For every Azumaya algebra $A$ over $X$ the set
\[
X(\A)^A:=\bigg\{(x_v)_v\in X(\A)\;\bigg|\;\sum_v\inv_v(A_v(x_v))=0\bigg\}
\]
contains $\overline{X(K)}$.

Hence, for each $S\subsetneq\Omega$ the set $X(\A_S)^A\subseteq X(\A_S)$ defined as
\[
\bigg\{(x_v)_{v\in\Omega\setminus S}\in X(\A_S)\;\bigg|\;\exists(x_v)_{v\in S}\in X(\A_{\Omega\setminus S})\textnormal{ such that }\sum_v\inv_v(A_v(x_v))=0\bigg\}
\]
contains $\overline{X(K)}$.
\end{theorem}

\section{Azumaya algebra}

In this section, we will define an Azumaya algebra over the scheme defined by \cref{zkeq}. We will then compute its local invariants.

\subsection{Construction}\label{azuconst}

Let $K=\Q$. Recall, that for each prime $p$ and $x\in\Q_p^\times$ we defined $r_p(x):=\frac x{p^{v_p(x)}}$.

For each $v\in\Omega$ let $(\cdot,\cdot):\Q_v^\times\times\Q_v^\times\rightarrow\{\pm1\}$ denote the Hilbert symbol of degree 2 (i.e., $(a,b)=1$ if and only if there exist $x,y\in\Q_v$ such that $a=y^2-bx^2$).

For each ring $R$ of characteristic different from 2 and $a,b\in R^\times$ let $\leg{a,b}{R}$ denote the quaternion algebra over $R$ with parameters $a,b$ (i.e., it is a free $R$-module with basis $1,i,j,ij$ such that $i^2=a$, $j^2=b$ and $ji=-ij$).

\begin{lemma}\label{quathilblemma}
For all $a\in\Q_v^\times$, we have:
\begin{align*}
(a,-1)=1&\Leftrightarrow\exists x,y\in\Q_v:a=x^2+y^2\\
&\Leftrightarrow\inv_v\leg{a,-1}{\Q_v}=0\\
&\Leftrightarrow\begin{cases}
a>0,&v=\infty,\\
r_2(a)\equiv1\mod4,&v=2,\\
0=0,&v\equiv1\mod4,\\
2\mid v_v(a),&v\equiv3\mod4.
\end{cases}
\end{align*}
For $v$-adic integers $a\in\Z_v^\times$, we even have:
\[
(a,-1)=1\Leftrightarrow\exists x,y\in\Z_v:a=x^2+y^2.
\]
\end{lemma}
\begin{proof}
See \cite[III.1, Theorem 1]{serre} and \cite[Proposition 1.1.7]{gille-tamas}.

The last equivalence is trivial if $v=\infty$, so let $p:=v$ be prime. The implication from right to left is obvious. Conversely, remark that there are at least $x',y'\in\Q_v$ such that $a=x'^2+y'^2$. Let $t:=\max(-v_p(x'),-v_p(y'))$. If $t\leq0$, then $x',y'\in\Z_v$, so assume $t>0$. Then we have $(x'p^t)^2+(y'p^t)^2=ap^{2t}$ with $x'p^t,y'p^t\in\Z_v$. Therefore, $(x'p^t)^2\equiv-(y'p^t)^2\mod p^2$. As $x'p^t$ and $y'p^t$ are not both divisible by $p$, this implies that $-1$ is a quadratic residue modulo $p^2$. Hence, $p\equiv1\mod4$, so there are $x'',y''\in\Z$ such that $p=x''^2+y''^2$. According to the pigeonhole principle, $r_p(a)\mod p$ is the sum of two quadratic residues, which we can lift to $x''',y'''\in\Z_p$ satisfying $r_p(a)\equiv x'''^2+y'''^2$ using Hensel's Lemma (as $p\neq2$ and $p\nmid r_p(a)$). Finally, repeated application of Brahmagupta's identity
\[
(xx''-yy'')^2+(xy''+yx'')^2=(x^2+y^2)(x''^2+y''^2)=p(x^2+y^2)
\]
yields $x,y\in\Z_p$ such that $a=p^{v_p(a)}r_p(a)=x^2+y^2$.
\end{proof}

Let $n\in\Z\setminus\{0\}$ and $a,b>0$ be integers such that $n>0$ or $2\nmid ab$. Consider the equation
\begin{equation}\label{zabnagl}
x^2+y^2+z^{ab}=n^a.
\end{equation}
Let
\[\X:=\Spec\Z[X,Y,Z]/(X^2+Y^2+Z^{ab}-n^a)
\]
and
\[\mathfrak X_\Q:=\X\otimes_\Z\Q=\Spec\Q[X,Y,Z]/(X^2+Y^2+Z^{ab}-n^a).
\]

The variety $\X_\Q$ is covered by the principal open subsets $U_1:=D(n-Z^b)\subseteq \mathfrak X_\Q$ and $U_2:=D(n^{a-1}+n^{a-2}Z^b+\dots+nZ^{(a-2)b}+Z^{(a-1)b})\subseteq \mathfrak X_\Q$. Indeed, as $a,n\in\Q^\times$, we have
\begin{align*}
V(n-Z^b)\cap V(n^{a-1}+\dots+Z^{(a-1)b})
&=V(n-Z^b,n^{a-1}+\dots+Z^{(a-1)b})\\
&=V(n-Z^b,an^{a-1})
=\emptyset.
\end{align*}

Consider the $\O_{\mathfrak X_\Q}|_{U_1}$-algebra 
\[
A_1:=\leg{n-Z^b,-1}{\O_{\mathfrak X_\Q}(U_1)}^\sim
\]
and the $\O_{\mathfrak X_\Q}|_{U_2}$-algebra
\[
A_2:=\leg{n^{a-1}+\dots+Z^{(a-1)b},-1}{\O_{\mathfrak X_\Q}(U_2)}^\sim.
\]

There is an $\O_{\mathfrak X_\Q}|_{U_1\cap U_2}$-algebra isomorphism $A_1|_{U_1\cap U_2}\xrightarrow\sim A_2|_{U_1\cap U_2}$ induced by
\begin{align*}
i&\mapsto\frac{Xi'+Yi'j'}{n^{a-1}+\dots+Z^{(a-1)b}}\\
j&\mapsto j'
\end{align*}
where $i,j$ and $i',j'$ are the canonical generators of $A_1|_{U_1\cap U_2}$ and $A_2|_{U_1\cap U_2}$, respectively (this is an $\O_{\mathfrak X_\Q}|_{U_1\cap U_2}$-algebra isomorphism as $X^2-(-1)Y^2=(n-Z^a)(n^{a-1}+\dots+Z^{(a-1)b})$).
Hence $A_1$ and $A_2$ can be glued along $U_1\cap U_2$ to obtain an $\O_{\mathfrak X_\Q}$-algebra $A$ such that $A|_{U_1}\cong A_1$ and $A|_{U_2}\cong A_2$.
Quaternion algebras over fields (with nonzero arguments) are central simple algebras, so $A$ is an Azumaya algebra.

In the following, we are interested in strong approximation ``at $Z$'' away from $\infty$. To this end, we choose a suitable topology: In $\Q_v^3$ we equip the first two components (i.e., those belonging to the variables $X$ and~$Y$) with the trivial topology (sometimes called indiscrete topology) and the last one (i.e., that belonging to the variable~$Z$) with the usual topology on $\Q_v$. Accordingly, the sets $\mathfrak X(\A)\subseteq\prod_v\Q_v^3$, $\X(\Z_v)\subseteq\Z_v^3\subseteq\Q_v^3$, etc. obtain the induced topologies.

Strong approximation with respect to the usual topology (i.e., at $X$, $Y$ and $Z$) seems more difficult, as for fixed $p,r,s\in\Z^+$ the equation $(px+r)^2+y^2=s$ does not fulfill the integral Hasse principle (unlike the equation $x^2+y^2=s$).

If strong approximation ``at $Z$'' away from $\infty$ is not fulfilled, then strong approximation away from $\infty$ with respect to the usual topology is not fulfilled, either.
\begin{lemma}
Let
\[
U:=U_1\cap U_2=D(n^a-Z^{ab})
\]
and
\[
I_v:=\inv_v(A(\X(\Z_v))).
\]
for any $v\in\Omega$. Then $I_v\subseteq\{0,1/2\}$ and
\[
I_v=\inv_v(A(U(\Q_v)\cap\X(\Z_v))).
\]
For $(x,y,z)\in \mathfrak X(\Q_v)$ we have
\begin{align*}
\inv_v(A(x,y,z))=0 &\Leftrightarrow (n-z^b,-1)=1&&\textnormal{ if }n-z^b\neq0,\\
\inv_v(A(x,y,z))=0 &\Leftrightarrow (n^{a-1}+\dots+z^{(a-1)b},-1)=1&&\textnormal{ if }n^{a-1}+\dots+z^{(a-1)b}\neq0,
\end{align*}
\begin{equation*}
w\in I_v\Leftrightarrow\exists z\in\Z_v\textnormal{ such that }(n^a-z^{ab},-1)=1\textnormal{ and }(n-z^b,-1)=\begin{cases}1,&w=0,\\-1,&w=1/2.\end{cases}
\end{equation*}
\end{lemma}
\begin{proof}
The inclusion $I_v\subseteq\{0,1/2\}$ follows from the fact that quaternion algebras over fields $k$ have order 2 in the Brauer group $\Br(k)$.

The first two equivalences follow straight from the definition of $A$ and \Cref{quathilblemma}.

It is easy to see that $U(\Q_v)$ is dense in $\mathfrak X(\Q_v)$. Hence, $U(\Q_v)\cap \X(\Z_v)$ is dense in $\X(\Z_v)$ because $\X(\Z_v)$ is an open subset of $\X(\Q_v)$. As $\inv_v\circ A:\mathfrak X(\Q_v)\rightarrow\Q/\Z$ is locally constant (even in the topology chosen above!), this implies that $I_v=\inv_v(A(U(\Q_v)\cap\X(\Z_v)))$.

For each $z\in\Z_v$ satisfying $n^a-z^{ab}\neq0$, there exist $x,y\in\Z_v$ such that $x^2+y^2=n^a-z^{ab}$ if and only if ${(n^a-z^{ab},-1)=1}$. Together with the first equivalence and $I_v=\inv_v(A(U(\Q_v)\cap\X(\Z_v)))$ this proves the final one.
\end{proof}

\subsection{Place \texorpdfstring{$\infty$}{∞}}

\begin{lemma}\label{real}
We have $I_\infty=\{0\}$.
\end{lemma}
\begin{proof}
A real solution is $(0,0,\sqrt[b]n)$ (as $n>0$ or $2\nmid b$), so $I_\infty\neq\emptyset$.

If $(x,y,z)\in U(\R)$, then $n^a-z^{ab}=x^2+y^2\geq0$ and $n>0$ or $2\nmid a$, so $n\geq z^b$. As $(x,y,z)\in U_1(\R)$ we have $n\neq z^b$, i.e., $n-z^b>0$, so $(n-z^b,-1)=1$. Hence $\inv_\infty(A(x,y,z))=0$.
\end{proof}

\subsection{Place 2}

\begin{lemma}\label{a1I2}
If $a=1$ and $b$ is odd, then $I_2=\{0\}$.
\end{lemma}
\begin{proof}
Due to $(n^a-z^{ab},-1)=(n-z^b,-1)$ for $z\in\Z_2$ it is obvious that $I_2\subseteq\{0\}$. The set $I_2$ is nonempty as there is some odd $z\in\Z$ such that $n-z\equiv1\text{ or }2\mod8$ and this fulfills $n^a-z^{ab}\equiv n-z^b\equiv n-z\equiv1\text{ or }2\mod8$ (as $b$ and $z$ are odd), so $(n^a-z^{ab},-1)=1$.
\end{proof}

\begin{lemma}\label{2notempty}
If $r_2(n)^a\equiv1\mod4$, then $I_2\neq\emptyset$.
\end{lemma}
\begin{proof}
Let $z:=0$. Then $r_2(n^a-z^{ab})\equiv r_2(n)^a\equiv1\mod4$, so $(n^a-z^{ab},-1)=1$.
\end{proof}

\begin{lemma}\label{2bnmid}
If $a\geq2$ and $r_2(n)^a\equiv1\mod4$ and $b\mid v_2(n)+1$, then $I_2=\{0,1/2\}$.
\end{lemma}
\begin{proof}
Let $z_1:=0$ and $z_2:=2^{(v_2(n)+1)/b}$. Now $r_2(n^a-z_1^{ab})\equiv r_2(n)^a\equiv1\mod4$ and $r_2(n^a-z_2^{ab})\equiv r_2(r_2(n)^a-2^a)\equiv r_2(n)^a\equiv1\mod4$, so $(n^a-z_1^{ab},-1)=(n^a-z_2^{ab},-1)=1$.

Furthermore $r_2(n-z_1^b)\equiv r_2(n)\not\equiv r_2(n)-2\equiv r_2(r_2(n)-2)\equiv r_2(n-z_2^b)\mod4$, so $(n-z_1^b,-1)\neq(n-z_2^b,-1)$.
\end{proof}

\begin{lemma}\label{6mod8ne}
If $a\geq3$ and $b$ are odd and $n\equiv6\mod8$, then $1/2\in I_2$.
\end{lemma}
\begin{proof}
Let $z:=-1$. Then $n^a-z^{ab}\equiv n^a+1\equiv1\mod4$ and $n-z^b\equiv n+1\equiv3\mod4$, so $1/2\in I_2$.
\end{proof}

\begin{lemma}\label{6mod8}
If $a,b\geq3$ are odd and $n\equiv6\mod8$, then $I_2=\{1/2\}$.
\end{lemma}
\begin{proof}
We know $1/2\in I_2$ from the previous lemma.

Let $(x,y,z)\in U(\Q_2)\cap\X(\Z_2)$.

If $z$ is odd, then $n^{a-1}+\dots+z^{(a-1)b}\equiv nz^{(a-2)b}+z^{(a-1)b}\equiv2+1\equiv3\mod4$, so $(n^{a-1}+\dots+z^{(a-1)b},-1)=-1$.

If $z$ is even, then
\[
1\equiv r_2(n^a-z^{ab})\equiv r_2((n/2)^a-2^{a(b-1)}(z/2)^{ab})\equiv r_2(n/2)^a\equiv3\mod4
\]
yields a contradiction.
\end{proof}

\begin{lemma}\label{not6mod8}
If $a,b$ are odd and $n\not\equiv6\mod8$, then $0\in I_2$.
\end{lemma}
\begin{proof}
The values of $z$ in the following table fulfill $r_2(n^a-z^{ab})\equiv r_2(n-z^b)\equiv1\mod4$:

\begin{tabular}{c|c|c|c|c|c|c|c}
$n\mod8$&0&1&2&3&4&5&7\\\hline
$z$&-1&0&0&1&$-1$&3&5
\end{tabular}
\end{proof}

\subsection{Odd places}\label{oddplacessection}

\begin{lemma}\label{odd}
We have $0\in I_p$ for all odd primes $p$.
\end{lemma}
\begin{proof}
One of the numbers $n^a$ or $n^a-1$ is not divisible by $p$. Set $z:=0$ or $z:=1$, respectively. Then $n^a-z^{ab}$ and hence $n-z^b$ are not divisible by $p$, so $(n^a-z^{ab},-1)=(n-z^b,-1)=1$.
\end{proof}

Hence it is only interesting whether $1/2\in I_p$.

\begin{lemma}\label{1mod4}
We have $I_p=\{0\}$ for all primes $p\equiv1\mod4$.
\end{lemma}
\begin{proof}
The previous lemma implies $I_p\neq\emptyset$. Furthermore, we have $(t,-1)=1$ for all $t\in\Q_p^\times$.
\end{proof}

Hence only the case $p\equiv3\mod4$ is interesting, so let $p\equiv3\mod4$ be prime for the rest of \Cref{oddplacessection}.

\begin{lemma}\label{pfull}
If $2\mid a$ and $2\nmid v_p(n)$, then $I_p=\{0,1/2\}$.
\end{lemma}
\begin{proof}
Take $z_1:=1$ and $z_2:=0$. Then
\begin{align*}
(n^a-z_1^{ab},-1)&=\phantom{-}1&&\text{ (as $2\mid0=v_p(n^a-z_1^{ab})$)}\\
(n^a-z_2^{ab},-1)&=\phantom{-}1&&\text{ (as $2\mid av_p(n)=v_p(n^a-z_2^{ab})$)}\\
(n-z_1^b,-1)&=\phantom{-}1&&\text{ (as $2\mid0=v_p(n-z_1^b)$)}\\
(n-z_2^b,-1)&=-1&&\text{ (as $2\nmid v_p(n)=v_p(n-z_2^b)$)}.
\end{align*}
\end{proof}

Let $a,b$ be odd for the rest of \Cref{oddplacessection}.
Then the following lemma simplifies the analysis of $I_p$.

\begin{lemma}\label{vnvzb}
Let $z\in\Z_p$ such that $n^a-z^{ab}\neq0$. Then the following statements are equivalent:
\begin{enumerate}[a)]
\item There are $x,y\in\Z_p$ such that $(x,y,z)\in \X(\Z_p)$ and $\inv_v(A(x,y,z))=1/2$ (hence $1/2\in I_p$).
\item $v_p(n)=v_p(z^b)$ and
\begin{align*}
1&\equiv v_p(r_p(n)^{a-1}+\dots+r_p(z)^{(a-1)b})&\mod2\\
1+v_p(n)&\equiv v_p(r_p(n)-r_p(z)^b)&\mod2\\
v_p(n)&\equiv v_p(r_p(n)^a-r_p(z)^{ab})&\mod2&.
\end{align*}
\end{enumerate}
\end{lemma}
\begin{remark}
The sum of the first two congruences in statement b) is the third one, so only two of them have to be proved.
\end{remark}
\begin{proof}[Proof of the lemma]
Assume a). Then (as $(n^{a-1}+\dots+z^{(a-1)b},-1)=-1$):
\[
2\nmid v_p(n^{a-1}+\dots+z^{(a-1)b})
\]

If $v_p(n)<v_p(z^b)$, then $2\mid(a-1)v_p(n)=v_p(n^{a-1}+\dots+z^{(a-1)b})$ yields a contradiction.

If $v_p(n)>v_p(z^b)$, then $2\mid(a-1)v_p(z^b)=v_p(n^{a-1}+\dots+z^{(a-1)b})$ yields a contradiction.

Hence $v_p(n)=v_p(z^b)$.

Then
\begin{align*}
1&\equiv v_p(n^{a-1}+\dots+z^{(a-1)b})\\
&\equiv(a-1)v_p(n)+v_p(r_p(n)^{a-1}+\dots+r_p(z)^{(a-1)b})\\
&\equiv v_p(r_p(n)^{a-1}+\dots+r_p(z)^{(a-1)b})\mod2
\end{align*}
and (as $(n-z^b,-1)=-1$)
\begin{align*}
1+v_p(n)
&\equiv v_p(n)+v_p(n-z^b)\\
&\equiv 2v_p(n)+v_p(r_p(n)-r_p(z)^b)\\
&\equiv v_p(r_p(n)-r_p(z)^b)\mod2.
\end{align*}
Conversely, b) implies
\begin{align*}
v_p(n^a-z^{ab})
&\equiv av_p(n)+v_p(r_p(n)^a-r_p(z)^{ab})\\
&\equiv(a+1)v_p(n)\equiv0\mod2,
\end{align*}
so there are $x,y\in\Z_p$ such that $(x,y,z)\in \X(\Z_p)$.

Furthermore
\[
v_p(n-z^b)\equiv v_p(n)+v_p(r_p(n)-r_p(z)^b)\equiv1\mod2,
\]
so $\inv_v(A(x,y,z))=1/2$.
\end{proof}

\begin{lemma}\label{jagypre1}
Assume $p\nmid an$. Then $1/2\not\in I_p$.
\end{lemma}
\begin{proof}
Suppose $(x,y,z)\in U(\Q_p)\cap\X(\Z_p)$ and $\inv_p(A(x,y,z))=1/2$. Then $v_p(n-z^b)$ and $v_p(n^{a-1}+\dots+z^{(a-1)b})$ are odd. Hence $n-z^b$ and $n^{a-1}+\dots+z^{(a-1)b}$ have to be divisible by $p$, so together $p\mid an^{a-1}$. Therefore $p\mid an$.
\end{proof}

\begin{lemma}\label{jagypre2}
Assume $b\nmid v_p(n)$. Then $1/2\not\in I_p$.
\end{lemma}
\begin{proof}
Suppose $(x,y,z)\in U(\Q_p)\cap\X(\Z_p)$ and $\inv_p(A(x,y,z))=1/2$. According to \Cref{vnvzb} we have $v_p(n)=v_p(z^b)$, so $b\mid v_p(n)$.
\end{proof}

\begin{lemma}\label{12insidenmid}
Let $p\nmid ab$. Then the following two statements are equivalent:
\begin{enumerate}[a)]
\item $1/2\in I_p$.
\item $b\mid v_p(n)$ and $2\nmid v_p(n)$ and there is some $z'\in\Z$ such that $p\mid r_p(n)^{a-1}+\dots+z'^{(a-1)b}$.
\end{enumerate}
\end{lemma}
\begin{proof}
Assume a). Let $(x,y,z)\in U(\Q_p)\cap\X(\Z_p)$ such that $\inv_p(A(x,y,z))=1/2$.

\Cref{vnvzb} shows that $v_p(n)=v_p(z^b)$ (so $b\mid v_p(n)$).
It also shows that $2\nmid v_p(r_p(n)^{a-1}+\dots+r_p(z)^{(a-1)b})$, so
$p\mid r_p(n)^{a-1}+\dots+r_p(z)^{(a-1)b}$.

If $2\mid v_p(n)$, then $2\nmid v_p(r_p(n)-r_p(z)^b)$ according to \Cref{vnvzb}. Therefore ${p\mid r_p(n)-r_p(z)^b}$ and $p\mid r_p(n)^{a-1}+\dots+r_p(z)^{(a-1)b}$. Together $p\mid a r_p(n)^{a-1}$, which is obviously impossible.
Therefore $2\nmid v_p(n)$.

Conversely, assume b). As $p\nmid ab$ and obviously $p\nmid z'$, we have
\[r_p(n)^a-z'^{ab}\not\equiv r_p(n)^a-z'^{ab}-abz'^{ab-1}p\equiv r_p(n)^a-(z'+p)^{ab}\mod p^2.
\]
Hence $v_p(r_p(n)^a-z'^{ab})\leq 1$ or $v_p(r_p(n)^a-(z'+p)^{ab})\leq1$. Let $z:=p^{v_p(n)/b}z'$ or $z:=p^{v_p(n)/b}(z'+p)$, respectively.

Therefore (as $p\mid r_p(n)^{a-1}+\dots+r_p(z)^{(a-1)b}$)
\[
1\leq v_p(r_p(n)^{a-1}+\dots+r_p(z)^{(a-1)b})\leq v_p(r_p(n)^a-r_p(z)^{ab})\leq1,
\]
so $v_p(r_p(n)^{a-1}+\dots+r_p(z)^{(a-1)b})=1$ and $v_p(r_p(n)^a-r_p(z)^{ab})\equiv1\equiv v_p(n)\mod2$.

Then \Cref{vnvzb} (together with its remark) shows that $1/2\in I_p$.
\end{proof}

\section{Failure of strong approximation and the integral Hasse principle}

In this section, we use the computations of local invariants of the Azumaya algebra $A$ over $\X_\Q$ defined in the previous section to obtain counterexamples to strong approximation and the integral Hasse principle.

\begin{lemma}
Let $a,b$ be odd. Then \cref{zabnagl} has $v$-adic integral solutions for each place~$v$.
\end{lemma}
\begin{proof}
See \Cref{real,odd,6mod8ne,not6mod8,a1I2}.
\end{proof}

The following theorem explains failures of strong approximation away from $\infty$. Not all $I_v$ have to be explicitly known to be able to apply it.

\begin{theorem}\label{notconstthm}
If $\X(\Z_v)\neq\emptyset$ for each $v\in\Omega$ and if $|I_w|=2$ for some $w\in\Omega$, then strong approximation ``at $Z$'' away from $\infty$ fails for the equation (\ref{zabnagl}) due to a Brauer--Manin obstruction.
\end{theorem}
\begin{proof}
Let $L_v=L_v'\in \X(\Z_v)$ for all $v\in\Omega\setminus\{w\}$ and $L_w,L_w'\in \X(\Z_w)$ such that $\inv_w(A(L_w))\neq\inv_w(A(L_w'))$. Then $\sum_{v\in\Omega}\inv_v(A(L_v))\neq\sum_{v\in\Omega}\inv_v(A(L_v'))$. Hence $(L_v)_v$ or $(L_v')_v\not\in \X(\A)^A$, i.e., $(L_v)_v$ or $(L_v')_v\not\in\overline{\X(\Q)}$ (where the closure is taken with respect to the topology defined in \Cref{azuconst}) although $(L_v)_v,(L'_v)_v\in \X(\A_{\{\infty\}})$.
\end{proof}

\begin{corollary}
If $a\geq2$ and $r_2(n)^a\equiv1\mod4$ and $b\mid v_2(n)+1$, then strong approximation ``at $Z$'' away from $\infty$ fails for (\ref{zabnagl}) due to a Brauer--Manin obstruction.
\end{corollary}
\begin{proof}
$\X(\Z_v)\neq\emptyset$ for all $v\neq2$ according to \Cref{real,odd}. $|I_2|=2$ according to \Cref{2bnmid}.
\end{proof}

\begin{corollary}[cf.\ \cref{nweakgeneralm}]\label{nweakgeneral}
According to the previous corollary the following equations do not fulfill strong approximation ``at $Z$'' away from $\infty$:
\begin{align*}
x^2+y^2+z^a=n^a,&\qquad\text{$a\geq3$ odd, $n\equiv1\mod4$}\\
x^2+y^2+z^a=n^a,&\qquad\text{$a\geq2$ even, $n>0$}
\end{align*}
\end{corollary}

\begin{remark}
Dietmann and Elsholtz showed in \cite{dietmann} and for the case $a=4$ in \cite{dietmannpub} that for $a\geq2$ and sufficiently large $N$ the number of integers $0<m\leq N$ such that $x^2+y^2+z^a=m$ does not fulfill strong approximation ``at $Z$'' away from $\infty$ is at least
\[
\begin{cases}
\frac{aN^{1/(2a)}}{\varphi(a)\log(N)},&a\text{ odd}\\
\frac{N^{1/2}}{2\log(N)},&a\text{ even}
\end{cases}
\]
The above example shows that this number is at least
\[
\begin{cases}
\frac{[N^{1/a}]+3}4,&a\text{ odd}\\
[N^{1/a}],&a\text{ even}
\end{cases}
\]
The following corollary gives a better estimate for even $a$.
\end{remark}

\begin{corollary}
If $n>0$ and $2\mid a$ and $n$ is not a sum of two squares, then strong approximation ``at $Z$'' away from $\infty$ fails for (\ref{zabnagl}) due to a Brauer--Manin obstruction.
\end{corollary}
\begin{proof}
There has to be some prime $p\equiv3\mod4$ such that $2\nmid v_p(n)$. Now $\X(\Z_v)\neq\emptyset$ for all $v\in\Omega$ according to \Cref{real,odd,2notempty} and $|I_p|=2$ according to \Cref{pfull}.
\end{proof}

Unfortunately, \Cref{notconstthm} cannot explain the overall absence of integral solutions ($\X(\A)^A\neq\emptyset$ whenever its conditions are satisfied) and it does not return any explicit points which are not contained in $\X(\A)^A$. To accomplish this, $I_v$ has to be explicitly computed for every $v\in\Omega$.

The following theorem is a generalization of the theorem in \cite{jagykaplansky} where an elementary proof for the case $a=b=3$ and $n=6q$ for primes $q\equiv1\mod4$ is given.

\begin{theorem}\label{jagythm}
Let $a,b\geq3$ be odd integers and $n\equiv6\mod8$ such that $b\nmid v_p(n)$ for all prime divisors $p\equiv3\mod4$ of $an$.

Then (\ref{zabnagl}) has no solutions in $\Z$ although it has $v$-adic integral solutions for each place $v$ and this is explained by a Brauer--Manin obstruction.

In particular, the integral Hasse principle fails.
\end{theorem}
\begin{proof}
We have $I_\infty=\{0\}$ according to \Cref{real} and $I_p=\{0\}$ for all primes $p\equiv1\mod4$ according to \Cref{1mod4}. Moreover, $I_2=\{1/2\}$ according to \Cref{6mod8}. Finally, $I_p=\{0\}$ for all primes $p\equiv3\mod4$ according to \Cref{odd,jagypre1,jagypre2}.

Hence there are $v$-adic integral solutions for each $v$ and $\sum_v\inv_v(A_v(x_v,y_v,z_v))=1/2$ for each $(x_v,y_v,z_v)_v\in\prod_v\X(\Z_v)$. This implies $\X(\Z)\subseteq \X(\A)^A\cap \X(\Z)=\emptyset$.
\end{proof}
\begin{remark}
For odd $a,b\geq3$ there are always infinitely many integers $n\equiv6\mod8$ such that $b\nmid v_p(n)$ for all prime divisors $p\equiv3\mod4$ of $an$, so there are infinitely many integers $n$ such that (\ref{zabnagl}) has no integral solutions. This confirms part b) of the remark following the Theorem in \cite{jagykaplansky}.
\end{remark}
\begin{proof}
Take $n:=2l\prod_{p\mid a}p$ where $l$ is the product of distinct primes such that $l\equiv1\mod4$ if $\prod_{p\mid a}p\equiv3\mod4$ and $l\equiv3\mod4$ otherwise.
\end{proof}

Dietmann and Elsholtz proved in \cite{dietmannpub} and \cite{dietmann} that (\ref{zabnagl}) does not fulfill strong approximation away from $\infty$ if
\begin{itemize}
\item $a=2$ and $n\equiv7\mod8$ is prime or
\item $a\geq3$ is odd, $b=1$ and $p\equiv1\mod4a$ is a prime such that $n=p^2$.
\end{itemize}

This can also be proved using the same strategy as above.

\section{Fulfillment of strong approximation}

Let $k\geq3$ be an odd integer and $m\in\Z\setminus\{0\}$.

Davenport and Heilbronn showed in \cite{davenportheilbronn}, that for all except $o(N)$ integers $1\leq m\leq N$ the equation
\begin{equation}\tag{1}
x^2+y^2+z^k=m\label{zk2}
\end{equation}
has a solution with $x,y,z\in\Z$.

As above, let
\[
\X:=\Spec\Z[X,Y,Z]/(X^2+Y^2+Z^k-m)
\]
and
\[
\X_\Q:=\X\otimes_\Z\Q=\Spec\Q[X,Y,Z]/(X^2+Y^2+Z^k-m)
\]
and
\[
U:=D(m-Z^k)\subseteq \X_\Q.
\]
Given $k$ and $m$ there may be multiple triples $(a,b,n)$ of integers with $a,b>0$ such that $k=ab$ and $m=n^a$, i.e., there may be multiple Azumaya algebras to consider for Brauer--Manin obstruction.
To this end, let $S(a,b,n):=\left(\prod_v\X(\Z_v)\right)^A$ with $A$ defined as in \Cref{azuconst} and let $I_v(a,b,n):=I_v$ as defined in \Cref{azuconst}.

Then we can define the subset $L$ of the solutions in $\A$ to \cref{zk2} for which there is no Brauer--Manin obstruction corresponding to any Azumaya algebra from \Cref{azuconst}:
\[
L:=\bigcap\limits_{{{\scriptstyle a,b,n\in\Z:\atop\scriptstyle a,b>0,}\atop\scriptstyle k=ab,}\atop\scriptstyle m=n^a}S(a,b,n)
\]
Of course, $\X(\Z)\subseteq L$.

The next theorem will show that Brauer--Manin obstructions with such Azumaya algebras explain all failures of strong approximation ``at $Z$'' away from $\infty$ if Schinzel's hypothesis (H) is true.

\begin{lemma}\label{rootdeg}
Let $K$ be a field, $k\geq1$ an odd integer and $u\in K$  such that $u$ is not a $p$-th power in $K$ for any prime divisor $p$ of $k$.
Then $[K(\sqrt[k]u):K]=k$.
\end{lemma}
\begin{proof}
See \cite[Thm.\ VI.9.1]{lang}.
\end{proof}

\begin{lemma}\label{phiirred}
Let $d,b\geq1$ be odd integers and $n\in\Q$ such that $n$ is not a $p$-th power for any prime divisor $p$ of $b$.

Then $\phi_d(X^b/n)\in\Q[X]$ is irreducible (where $\phi_d$ is the $d$-th cyclotomic polynomial).
\end{lemma}
\begin{proof}
For any positive integer $s$, let $\zeta_s$ denote a primitive $s$-th root of unity. The polynomial $\phi_d(X^b/n)$ has a root $\sqrt[b]n\cdot\zeta_{bd}$. Assume $\phi_d(X^b/n)$ is reducible. Hence (as $\deg(\phi_d(X^b/n))=b\varphi(d)$)
\begin{align*}
[\Q(\sqrt[b]n\cdot\zeta_{bd}):\Q(\zeta_d)]\varphi(d)
&=[\Q(\sqrt[b]n\cdot\zeta_{bd}):\Q(\zeta_d)]\cdot[\Q(\zeta_d):\Q]\\
&=[\Q(\sqrt[b]n\cdot\zeta_{bd}):\Q]\\
&<b\varphi(d),
\end{align*}
so $[\Q(\sqrt[b]n\cdot\zeta_{bd}):\Q(\zeta_d)]<b$. \Cref{rootdeg} therefore implies that $n\cdot\zeta_d$ is a $p$-th power in $\Q(\zeta_d)$ for some prime divisor $p$ of $b$, say $x\in\Q(\zeta_d)$ and $x^p=n\cdot\zeta_d$.

If $p\mid d$, then $(x/\overline x)^p=\zeta_d^2$, but this is impossible as the roots of unity in $\Q(\zeta_d)$ are $\mu_{2d}$ but $\mu_{2d}^p=\mu_{2d/p}\not\ni\zeta_d^2$.

Hence $p\nmid d$. Let then $r\in\Z$ such that $rp\equiv1\mod d$. Hence for $y:=x/\zeta_d^r$ we have
\[
y^p=\frac{x^p}{\zeta_d^{rp}}=\frac{n\cdot\zeta_d}{\zeta_d}=n,
\]
so in particular every $\tau\in\Gal(\Q(\zeta_d)|\Q)$ fulfills $(\tau y)^p=\tau y^p=y^p$. Therefore $(\tau y/y)^p=1$ but this is only possible if $\tau y=y$ as $\Q(\zeta_d)$ contains no primitive $p$-th root of unity. Hence we conclude that $y\in\Q$. This yields a contradiction as $y^p=n$ but $n$ is not a $p$-th power in $\Q$ by assumption.
\end{proof}

Recall the statement of Schinzel's hypothesis (H):
\begin{hypothesis}[H]
Let $f_1,\dots,f_s\in\Z[X]$ be polynomials irreducible in $\Q[X]$ such that
\[\gcd\{f_1(x)\cdots f_s(x)\mid x\in\Z\}=1
\]
and $f_i(x)\rightarrow\infty$ for $x\rightarrow\infty$ for each $1\leq i\leq s$. Then there is some $x\in\Z$ such that $f_i(x)$ is prime for each $i$.
\end{hypothesis}

Below, we will use the following consequence of Schinzel's hypothesis (H).

\begin{lemma}\label{schinzelhelp}
Let $f_1,\dots,f_s\in\Z[X]$ be polynomials irreducible in $\Q[X]$ such that 
\[\gcd\{f_1(x)\cdots f_s(x)\mid x\in\Z\}=1
\]
and $f_i(x)\rightarrow\infty$ for $x\rightarrow\infty$ for each $1\leq i\leq s$. Let furthermore $c,e\in\Z$ such that $v_p(f_i(c))\leq v_p(e)$ for all prime divisors $p$ of $e$ and each $i$. Assume Schinzel's hypothesis~(H) is true. Then there is some $x\equiv c\mod e$ such that $\frac{f_i(x)}{\gcd(f_i(c),e)}$ is prime for each $i$.
\end{lemma}
\begin{proof}
Let $g_i(X):=\frac{f_i(eX+c)}{\gcd(f_i(c),e)}$. Obviously $g_i\in\Z[X]$ and $g_i$ is irreducible in $\Q[X]$.

Assume that $p$ is a prime divisor of $\gcd\{g_1(x)\cdots g_s(x)\mid x\in\Z\}$. There must be some $y\in\Z$ such that $p\nmid f_1(y)\cdots f_s(y)$.

For $p\nmid e$ there is some $r\in\Z$ such that $er+c\equiv y\mod p$. Then
\[p\mid g_1(r)\cdots g_s(r)\mid f_1(er+c)\cdots f_s(er+c),
\]
so $0\equiv f_1(er+c)\cdots f_s(er+c)\equiv f_1(y)\cdots f_s(y)\mod p$ yields a contradiction.

For $p\mid e$ we have $v_p(f_i(c))\leq v_p(e)$ and hence $p\nmid \frac{f_i(c)}{\gcd(f_i(c),e)}=g_i(0)$ for each $i$. Therefore $p\nmid g_1(0)\cdots g_s(0)$, which is a contradiction too.

Hence $\gcd\{g_1(x)\cdots g_s(x)\mid x\in\Z\}=1$ and $g_i\in\Z[X]$ is irreducible in $\Q[X]$ and $g_i(x)\rightarrow\infty$ for $x\rightarrow\infty$ for each $i$, so Schinzel's hypothesis (H) implies that there is some $z\in\Z$ such that $g_i(z)=\frac{f_i(ez+c)}{\gcd(f_i(c),e)}$ is prime for each $i$. The claim follows with $x:=ez+c$.
\end{proof}

\begin{theorem}[cf.\ \cref{schinzelthmm}]\label{schinzelthm}
If Schinzel's hypothesis (H) is true, then $\overline{\X(\Z)}=L$ (where the closure is taken with respect to the topology defined in \Cref{azuconst}).
\end{theorem}
\begin{proof}
Let $a$ be the largest divisor of $k$ such that $m$ is an $a$-th power. Let $n:=\sqrt[a]m$ and $b:=\frac ka$.
Consider the factorization\footnote{From now on, ``divisor'' will mean ``\emph{positive} divisor''.}
\begin{align*}
X^{ab}+n^a
&=-n^a\left(\left(-\frac{X^b}n\right)^a-1\right)
=-n^a\prod_{d\mid a}\phi_d\left(-\frac{X^b}n\right)\\
&=\prod_{d\mid a}(-n)^{\varphi(d)}\phi_d\left(-\frac{X^b}n\right).
\end{align*}
The last equality follows from the fact that $\sum_{d\mid a}\varphi(d)=a$.

Let $f_d(X):=(-n)^{\varphi(d)}\phi_d\left(-\frac{X^b}n\right)$ for each divisor $d$ of $a$.

The polynomials $f_d(X)\in\Q[X]$ are irreducible according to \Cref{phiirred} and the choice of $a$.
Moreover $f_d(X)\in\Z[X]$ as $\phi_d(Y)$ has integral coefficients and degree $\varphi(d)$.

Take $(x_v,y_v,z_v)_v\in L$, a finite set $T\subseteq\Omega\setminus\{\infty\}$ and $t\geq0$. We have to show that there is some $(x,y,z)\in \X(\Z)$ such that $v_p(z_p-z)\geq t$ for all $p\in T$.

The set $L$ is open, as it is the intersection of finitely many sets $S(a,b,n)$, which are themselves open as the map $\prod_v\X(\Z_v)\rightarrow\{0,1/2\}$ given by $(x_v,y_v,z_v)_v\mapsto\sum_v\inv_v(A(x_v,y_v,z_v))$ is locally constant. As $U(\Q_v)\cap \X(\Z_v)$ is dense in $\X(\Z_v)$ for each $v\in\Omega$, we can hence assume that $n^a-z_p^{ab}\neq0$ for all primes $p$.

{\parindent=1cm
\hangindent=1cm
\textbf{Claim.} For each $d\mid a$ we have $\prod_p(f_d(-z_p),-1)=1$ where the product runs over all primes $p$ (in particular $(f_d(-z_p),-1)=1$ for almost all primes $p$).

\emph{Proof.} The factors $(f_d(-z_p),-1)$ are well-defined as $f_d(-z_p)\neq0$ due to $n^a-z_p^{ab}\neq0$.

As $(x_v,y_v,z_v)_v\in S(\frac ad,db,n^d)$, we conclude that $\prod_p(n^d-z_p^{db},-1)=1$. But
\[
X^{db}+n^d=\prod_{d'\mid d}(-n)^{\varphi(d')}\phi_{d'}\left(-\frac{X^b}n\right)=\prod_{d'\mid d}f_{d'}(X),
\]
so
\[
1=\prod_p(n^d-z_p^{db},-1)=\prod_p\prod_{d'\mid d}(f_{d'}(-z_p),-1).
\]

The result follows by induction by $d$.
\qed
}

Assume without loss of generality that $2\in T$ and that for each prime $p$: if $p\in T$, then $t\geq v_p(f_d(-z_p))+2$ for all $d\mid a$ and if $(f_d(-z_p),-1)\neq1$ for some $d\mid a$, then $p\in T$.

The leading coefficient of $f_d(X)$ is 1 and its degree is $b\varphi(d)>0$, so $f_d(u)\rightarrow\infty$ for $u\rightarrow\infty$.
Furthermore $\gcd\{\prod_{d\mid a}f_d(u)\mid u\in\Z\}=\gcd\{u^{ab}+n^a\mid u\in\Z\}=1$, as it divides $\gcd(n^a,1+n^a)=1$.

The Chinese remainder theorem shows that there is some $c\in\Z$ such that $c\equiv -z_p\mod p^t$ for each $p\in T$.

Let $e:=\prod_{p\in T}p^t$.
Now $v_p(f_d(c))=v_p(f_d(-z_p))\leq t-2<t=v_p(e)$ for each $p\in T$ and each $d\mid a$.
Hence applying \Cref{schinzelhelp} proves that there is some $z\in\Z$ such that $-z\equiv c\mod e$ (in particular $v_p(z_p-z)\geq\min(v_p(z_p+c),v_p(-z-c))\geq t$ for each prime $p\in T$) and $\frac{f_d(-z)}{\gcd(f_d(c),e)}$ is prime for each $d\mid a$.

Therefore $f_d(-z)\equiv f_d(c)\equiv f_d(-z_p)\mod p^t$, so in particular
\[
v_p(f_d(-z))=v_p(f_d(-z_p))
\]
and
\[
r_p(f_d(-z))\equiv r_p(f_d(-z_p))\mod p^2
\]
as $t\geq v_p(f_d(-z_p))+2$ for each $p\in T$.

Hence $r_2(f_d(-z))\equiv r_2(f_d(-z_2))\mod4$. As $\prod_p(f_d(-z_p),-1)=1$ and moreover ${p\in T}$ whenever $(f_d(-z_p),-1)\neq1$ (i.e., whenever $p^{v_p(f_d(-z_p))}\not\equiv1\mod4$), the following congruence holds:
\[
r_2(f_d(-z_2))\equiv\prod_{p\neq2}p^{v_p(f_d(-z_p))}\equiv\prod_{p\in T\setminus\{2\}}p^{v_p(f_d(-z_p))}\equiv r_2(\gcd(f_d(c),e))\mod4.
\]
Together we get $r_2(f_d(-z))\equiv r_2(\gcd(f_d(c),e))\mod4$, so $r_2\left(\frac{f_d(-z)}{\gcd(f_d(c),e)}\right)\equiv1\mod4$. As $\frac{f_d(-z)}{\gcd(f_d(c),e)}$ is prime, it is therefore a sum of two squares.

The product
\[
\prod_{d\mid a}\gcd(f_d(c),e)=\prod_{p\in T}\prod_{d\mid a}p^{v_p(f_d(-z_p))}=\prod_{p\in T}p^{v_p(n^a-z_p^{ab})}
\]
is also a sum of two squares as all primes $p\equiv3\mod4$ occur an even number of times in it because they do so in $n^a-z_p^{ab}=x^2+y^2$.

Now in the factorization
\[
n^a-z^{ab}=\prod_{d\mid a}f_d(-z)=\Bigg(\prod_{d\mid a}\gcd(f_d(c),e)\Bigg)\cdot\Bigg(\prod_{d\mid a}\frac{f_d(-z)}{\gcd(f_d(c),e)}\Bigg)
\]
the first product and each factor of the second product are sums of two squares, so $n^a-z^{ab}$ is a sum of two squares, too.
\end{proof}
\begin{lemma}\label{S1kn}
We have $S(1,k,m)=\prod_v\X(\Z_v)$.
\end{lemma}
\begin{proof}
For each place $v$ and $(x_v,y_v,z_v)\in U(\Q_v)\cap\X(\Z_v)$ we have
\[(m-z_v^k,-1)=(m^1-z_v^{1\cdot k},-1)=1,
\]
so $\inv_v(A(x_v,y_v,z_v))=0$.
\end{proof}
\begin{remark}
\Cref{schinzelthm} does not hold for arbitrary even $k$. For example \cref{zk2} does not have an integral solution for $k=4$ and $m=22$ but it has local solutions $(x_v,y_v,z_v)_v\in \prod_v\X(\Z_v)=S(1,k,m)=L$ given by
\[
z_v=
\begin{cases}
1,&v=2\text{ or }11\\
0,&\text{else}.
\end{cases}
\]
\end{remark}

\begin{corollary}\label{bunyathm}
Assume Schinzel's hypothesis (H) is true.
Let $k$ be an odd positive integer and assume that $m$ is not a $p$-th power for any prime $p\mid k$. Then there exists an integral solution to \cref{zk2}.
\end{corollary}
\begin{proof}
Then $\overline{\X(\Z)}=L=S(1,k,n)=\prod_v\X(\Z_v)\neq\emptyset$ according to \Cref{real,odd,a1I2}.
\end{proof}
\begin{remark}
The proof of \Cref{bunyathm} needs Schinzel's hypothesis (H) only in the case of one polynomial, also known as Bunyakovsky's conjecture (cf. \cite[Conjecture 1]{moroz}).

Under this assumption, this corollary includes the result of Davenport and Heilbronn mentioned above.
\end{remark}
\begin{corollary}[cf.\ \cref{primekm}]\label{primek}
Assume Bunyakovsky's conjecture is true.

Let $k$ be an odd prime. Then there exists an integral solution to \cref{zk2}.
\end{corollary}
\begin{proof}
If $m$ is a $k$-th power, then $(0,0,\sqrt[k]m)$ is a solution. Otherwise the previous corollary applies.
\end{proof}

\begin{lemma}\label{abgeneral}
Let $k$ be the product of two odd primes $a$ and $b$ and let $m\in\Z\setminus\{0\}$. 

For the existence of integral solutions to \cref{zk2}, it is necessary and under Schinzel's hypothesis (H) also sufficient that the following two statements are both true.
\begin{itemize}
\item There is no $n\in\Z$ such that $m=n^a$ and $0\not\in I_2(a,b,n)$ and $1/2\not\in I_p(a,b,n)$ for each prime $p\equiv3\mod4$.
\item There is no $n\in\Z$ such that $m=n^b$ and $0\not\in I_2(b,a,n)$ and $1/2\not\in I_p(b,a,n)$ for each prime $p\equiv3\mod4$.
\end{itemize}
\end{lemma}
\begin{proof}
For necessity, let $n\in\Z$ and $m=n^a$ and $0\not\in I_2(a,b,n)$ and $1/2\not\in I_p(a,b,n)$ for each prime $p\equiv3\mod4$. Then $S(a,b,n)=\emptyset$ according to \Cref{real,1mod4}. Hence $\X(\Z)\subseteq S(a,b,n)=\emptyset$.

Conversely, if $m$ is an $ab$-th power, then $(0,0,\sqrt[ab]m)$ is a solution.

If $m$ is neither an $a$-th power nor a $b$-th power, then \Cref{bunyathm} proves the claim.

Let therefore without loss of generality $m$ be an $a$-th power but not an $ab$-th power, so there is some $n\in\Z$ such that $m=n^a$. Now the first statement given above implies that $0\in I_2(a,b,n)$ or $1/2\in I_p(a,b,n)$ for some prime $p\equiv3\mod4$. Together with \Cref{real,odd,6mod8,not6mod8} this shows that $S(a,b,n)\neq\emptyset$.

Moreover $S(1,ab,m)=\prod_v\X(\Z_v)$ according to \Cref{S1kn}.

Hence $\overline{\X(\Z)}=S(1,ab,m)\cap S(a,b,n)=S(a,b,n)\neq\emptyset$.
\end{proof}

\begin{theorem}[cf.\ \cref{abthmm}]\label{abthm}
Let $k$ be the product of two primes $a,b\equiv1\mod4$ and let $m\in\Z\setminus\{0\}$.

For the existence of integral solutions to \cref{zk2}, it is necessary and under Schinzel's hypothesis (H) also sufficient that the following two statements are both true.
\begin{itemize}
\item There is no $n\equiv6\mod8$ such that $m=n^a$ and for each prime $p\equiv3\mod4$ dividing $n$:

$b\nmid v_p(n)$ or $2\mid v_p(n)$ or there is no $z'\in\{0,\dots,p-1\}$ such that
\[p\mid r_p(n)^{a-1}+\dots+z'^{(a-1)b}.
\]
\item There is no $n\equiv6\mod8$ such that $m=n^b$ and for each prime $p\equiv3\mod4$ dividing $n$:

$a\nmid v_p(n)$ or $2\mid v_p(n)$ or there is no $z'\in\{0,\dots,p-1\}$ such that
\[p\mid r_p(n)^{b-1}+\dots+z'^{(b-1)a}.
\]
\end{itemize}
\end{theorem}
\begin{proof}
The condition of \Cref{abgeneral} is equivalent to the condition of this theorem according to \Cref{6mod8,not6mod8,jagypre1,12insidenmid}.
\end{proof}

\section{Algorithm}\label{algosection}

We give an algorithm to decide for given $m\in\Z$ and odd $k>0$ if the number $m$ is of the form $x^2+y^2+z^k$.

\begin{algorithmic}[1]
\Function{combi}{$a,b,n,p$}
  \State Consider all Hilbert symbols over $\Q_p$.\label{nots}
  \State Let $f_d(Z):=(-n)^{\varphi(d)}\phi_d(-Z^b/n)$.
  \State For each $z\in\Z$ let $w_z:=\{d\text{ divisor of }a\mid (f_d(-z),-1)\neq1\}$.
  \State For each $z\in\Z$ and $t\geq0$ let $G_{t,z}:=\{d\text{ divisor of }a\mid v_p(f_d(-z))+1\geq t\}$.\label{note}
  \If{$p\equiv1\mod4$}
    \State\Return$\{\emptyset\}$
  \Else
    \State $W\gets\emptyset$
    \State $S_0\gets\{0\}$
    \State $t\gets0$
    \While{$S_t\neq\emptyset$}
      \State $S_{t+1}\gets\emptyset$
      \ForAll{$z\in S_t$}
	\If{$(n^a-z^{ab},-1)=1$}
	  \State $W\gets W\cup\{w_z\}$
	\EndIf\label{Rpos1}
	\If{$|G_{t,z}|>1$ or\Statex[6] ($|G_{t,z}|=1$ and $w_z\setminus G_{t,z}\not\in W$ and $w_z\cup G_{t,z}\not\in W$)}
	  \State $S_{t+1}\gets S_{t+1}\cup\{z'\in[0,p^{t+1}-1]\mid z'\equiv z\mod p^t\}$
	\EndIf
      \EndFor
      \State $t\gets t+1$
    \EndWhile
    \State\Return $W$
  \EndIf
\EndFunction
\algstore{alg1}
\end{algorithmic}

\begin{lemma}
Let $a\geq1$ and $b\geq3$ be odd integers, $n$ an integer such that $n$ is not a $q$-th power for any prime divisor $q$ of $b$ (then $n^a-z^{ab}\neq0$ for all $z\in\Z$) and let $p$ be prime.

Let $f_d(Z)$, $w_z$ and $G_{t,z}$ be defined as in lines \ref{nots} to \ref{note}.

Then \Call{combi}{$a,b,n,p$} terminates and returns
\[
C:=\{w_z\mid z\in\{0,1,2,\dots\}\textnormal{ such that }(n^a-z^{ab},-1)=1\}.
\]
\end{lemma}
\begin{proof}
For $p\equiv1\mod4$ the result is immediate, so let $p\not\equiv1\mod4$.

The algorithm describes a pruned breadth-first search\footnote{i.e., a breadth-first search in which insignificant branches are ignored} on the infinite directed graph with the following node set $V$ and edge set $E$:
\[
V:=\{(t,z)\mid t\geq0\text{ and }0\leq z\leq p^t-1\}\]
\[E:=\{((t,z),(t+1,z'))\in V^2\mid z\equiv z'\mod p^t\}.
\]
Each time a node $(t,z)$ with $(n^a-z^{ab},-1)=1$ is visited, $w_z$ is appended to $W\subseteq C$.

In this graph there is a path from $(t,z)\in V$ to $(t',z')\in V$ if and only if $t\leq t'$ and ${z\equiv z'\mod p^t}$.

Obviously every node can be reached from $(0,0)$.
Therefore a complete breadth-first search would eventually find every $c\in C$.

Let $(t,z)$ and $(t',z')$ be nodes such that $(t',z')$ is reachable from $(t,z)$, i.e., such that ${z'\equiv z\mod p^t}$. If $d\not\in G_{t,z}$ is a divisor of $a$, then $v_p(f_d(-z))<t-1$. Then $f_d(-z')\equiv f_d(-z)\mod p^t$ implies $v_p(f_d(-z'))=v_p(f_d(-z))$ and $r_p(f_d(-z'))\equiv r_p(f_d(-z))\mod p^2$, so $(f_d(-z'),-1)=(f_d(-z),-1)$. Hence $w_{z'}\setminus G_{t,z}=w_z\setminus G_{t,z}$.

In particular $w_{z'}=w_z$ if $G_{t,z}=\emptyset$. Therefore the breadth-first search does not have to be continued from $(t,z)$ on if $G_{t,z}=\emptyset$.

Every set $c\in C$ has an even number of elements as $\prod_{d\mid a}(f_d(-z),-1)=(n^a-z^{ab},-1)$ for each $z\in\Z$. Hence for each subset $w'$ of the set of divisors of $a$ and each divisor $g$ of $a$ at least one of the sets $w'\setminus\{g\}$ and $w'\cup\{g\}$ is not contained in $C$.

Therefore, if $G_{t,z}=\{g\}$ for some divisor $g$ of $a$ and $w_z\setminus\{g\}$ or $w_z\cup\{g\}$ has already been found (and is therefore contained in $C$), then the breadth-first search does not have to be continued from $(t,z)$, either.

Hence altogether the above algorithm finds every element of $C$, so the only remaining question is whether it terminates in a finite amount of time.

Assume it does not.
Then there has to be an infinite path $(0,z_0)\rightarrow(1,z_1)\rightarrow(2,z_2)\rightarrow\dots$ of which every edge is visited during the breadth-first search.
The definition of the edge set $E$ proves that $z_t$ converges to some $\overline z\in\Z_p$ such that ${\overline z\equiv z_t\mod p^t}$ for each $t\geq0$.

If $d\in G_{t+1,z_{t+1}}$, then $f_d(-z_t)\equiv f_d(-z_{t+1})\equiv0\mod p^t$, so $d\in G_{t,z_t}$.

Hence $G_{0,z_0}\supseteq G_{1,z_1}\supseteq G_{2,z_2}\supseteq\dots$.

If $G_{t,z_t}$ was empty for some $t$, then the breadth-first search would not continue from the node $(t,z_t)$ on.
Hence $\bigcap_{t\geq0}G_{t,z_t}\neq\emptyset$.

If $g\in\bigcap_{t\geq0}G_{t,z_t}$, then $f_g(-\overline z)\equiv f_g(-z_t)\equiv0\mod p^{t-1}$ for all $t\geq1$, so $f_g(-\overline z)=0$. However, the polynomials $f_d(Z)$ with $d\mid a$ are irreducible (according to \Cref{phiirred}) and pairwise distinct, so no two of them have any common roots. Hence $\bigcap_{t\geq0}G_{t,z_t}$ contains exactly one element $g$. Let $T\geq0$ such that $G_{T,z_T}=\{g\}$.
Then $w_{z_t}\setminus\{g\}=w_{z_T}\setminus\{g\}$ for each $t\geq T$.

As the breadth-first search continues at every node $(t,z_t)$, it follows that neither $w_{z_T}\setminus\{g\}$ nor $w_{z_T}\cup\{g\}$ are found. As every element of $C$ is eventually found, this shows that $w_{z_T}\setminus\{g\},w_{z_T}\cup\{g\}\not\in C$.

However,
\[n^a-(\overline z\pm p^s)^{ab}\equiv n^a-\overline z^{ab}\mp ab\overline z^{ab-1}p^s\equiv\mp ab\overline z^{ab-1}p^s\mod p^{2s}
\]
(as $f_g(-\overline z)=0$ and $f_g(Z)$ divides $n^a-Z^{ab}$). Therefore (as $\overline z\neq0$ due to $n^a-\overline z^{ab}=0$), by choosing $s$ sufficiently large and of the correct parity and the appropriate sign, we get some $z'\equiv\overline z\mod p^T$ such that $(n^a-z'^{ab},-1)=1$, so $w_{z'}\in C$. Moreover, $w_{z'}\setminus\{g\}=w_{z_T}\setminus\{g\}$ because of $G_{t,z_t}=\{g\}$. This proves that $w_{z_T}\setminus\{g\}\in C$ or $w_{z_T}\cup\{g\}\in C$, which is a contradiction.

Therefore the algorithm terminates in a finite amount of time.
\end{proof}

Let $\bigtriangleup$ denote the symmetric difference (i.e., $A\bigtriangleup B=(A\setminus B)\cup(B\setminus A)$).

Consider now the following algorithm:

\begin{algorithmic}[1]
\algrestore{alg1}
\Function{ispossible}{$k,m$}
  \If{$\sqrt[k]m\in\Z$}
    \State \Return true
  \Else
    \State $a\gets\max\{d\text{ divisor of }k\mid m\text{ is $d$-th power}\}$
    \State $b\gets \frac ka$
    \State $n\gets\sqrt[a]m$
    \State $T\gets\{\emptyset\}$
    \ForAll{$p\mid 2an$ prime}
      \State $W\gets$ \Call{combi}{$a,b,n,p$}
      \State $T\gets\{t\bigtriangleup w\mid t\in T, w\in W\}$
    \EndFor\label{Tfinished}
    \State \Return $\emptyset\in T$
  \EndIf
\EndFunction
\end{algorithmic}

\begin{theorem}
Let $k\geq1$ be odd and $m\in\Z$.

Then \Call{ispossible}{$k,m$} always terminates. If it returns ``false'', then $m$ is not of the form $x^2+y^2+z^k$ for integers $x,y,z$. If Schinzel's hypothesis (H) is true, then the converse also holds. 
\end{theorem}
\begin{proof}
The case $\sqrt[k]m\in\Z$ is obvious, so assume $\sqrt[k]m\not\in\Z$.

According to the previous lemma (and as $\{0,1,2,\dots\}$ is dense in $\Z_p$ and the map $\Q_p^\times\rightarrow\{\pm1\}$ defined by $u\mapsto (u,-1)$ is locally constant for each prime $p$), the set $T$ can after line \ref{Tfinished} be described as follows:
\[
T=\bigg\{\mathop{\bigtriangleup}\limits_{p\mid 2an}w_{z_p}\ \bigg|\ (x_p,y_p,z_p)_p\in\prod_{p\mid 2an}U(\Q_p)\cap\X(\Z_p)\bigg\}.
\]
Hence $\emptyset\in T$ if and only if there is some $(x_p,y_p,z_p)_p\in\prod_{p\mid 2an}U(\Q_p)\cap\X(\Z_p)$ such that $\prod_{p\mid 2an}(f_d(-z_p),-1)=1$ for each divisor $d$ of $a$.

If $v$ is a prime not dividing $2an$ or $v=\infty$ and $(x_v,y_v,z_v)\in U(\Q_v)\cap\X(\Z_v)$, then, according to \Cref{real,1mod4,jagypre1}, $(n^d-z_v^{db},-1)=1$ for each $d\mid a$.
Furthermore, for each such place $v$ the set $U(\Q_v)\cap\X(\Z_v)$ is nonempty according to \Cref{real,odd}.

The equation
\[
(n^d-z_v^{db},-1)=\prod_{d'\mid d}(f_{d'}(-z_v),-1)
\]
therefore shows (as in the proof of the claim in the proof of \Cref{schinzelthm}) that there is some $(x_v,y_v,z_v)_v\in\prod_vU(\Q_v)\cap\X(\Z_v)$ such that $\prod_v(n^d-z_v^{bd},-1)=1$ for each $d\mid a$ if and only if $\emptyset\in T$.
Therefore $L\neq\emptyset$ if and only if $\emptyset\in T$, so the claim follows with \Cref{schinzelthm}.
\end{proof}

For each odd composite integer $1<k<50$, \Cref{listexceptions} lists values of positive integers $m\leq10^9$ such that \cref{zk2} has no integral solution, determined using our algorithm. The lists might be incomplete if Schinzel's hypothesis (H) is false.

\renewcommand{\arraystretch}{1.17}
\begin{table}[ht]
\caption{List of integers without integral solution}\label{listexceptions}
\begin{tabular}{c|p{11.5cm}}
\textbf{$k$}&\textbf{List of integers $1\leq m\leq10^9$ without integral solution}\\\hline
9&$6^3$, $30^3$, $54^3$, $78^3$, $102^3$, $126^3$, $150^3$, $174^3$, $198^3$, $222^3$, $246^3$, $294^3$, $318^3$, $342^3$, $366^3$, $390^3$, $414^3$, $438^3$, $462^3$, $486^3$, $510^3$, $534^3$, $558^3$, $582^3$, $606^3$, $630^3$, $654^3$, $678^3$, $726^3$, $750^3$, $774^3$, $798^3$, $822^3$, $846^3$, $870^3$, $894^3$, $918^3$, $942^3$, $966^3$, $990^3$\\\hline
15&$6^3$, $30^3$, $54^3$, $78^3$, $102^3$, $126^3$, $150^3$, $174^3$, $198^3$, $222^3$, $246^3$, $270^3$, $294^3$, $318^3$, $342^3$, $366^3$, $390^3$, $414^3$, $438^3$, $462^3$, $510^3$, $534^3$, $558^3$, $582^3$, $606^3$, $630^3$, $654^3$, $678^3$, $702^3$, $726^3$, $750^3$, $774^3$, $798^3$, $822^3$, $846^3$, $870^3$, $894^3$, $918^3$, $942^3$, $966^3$, $990^3$,\hfill\par
$6^5$, $14^5$, $22^5$, $30^5$, $38^5$, $46^5$, $54^5$, $62^5$\\\hline
21&$6^3$, $30^3$, $54^3$, $78^3$, $102^3$, $126^3$, $150^3$, $174^3$, $198^3$, $222^3$, $246^3$, $270^3$, $294^3$, $318^3$, $342^3$, $366^3$, $390^3$, $414^3$, $438^3$, $462^3$, $486^3$, $510^3$, $534^3$, $558^3$, $582^3$, $606^3$, $630^3$, $654^3$, $678^3$, $702^3$, $726^3$, $750^3$, $774^3$, $798^3$, $822^3$, $846^3$, $870^3$, $894^3$, $918^3$, $942^3$, $966^3$, $990^3$,\hfill\par
$14^7$\\\hline
25&$6^5$, $14^5$, $22^5$, $30^5$, $38^5$, $46^5$, $54^5$, $62^5$\\\hline
27&$6^3$, $30^3$, $46^3$, $54^3$, $62^3$, $78^3$, $102^3$, $118^3$, $126^3$, $150^3$, $174^3$, $198^3$, $206^3$, $222^3$, $246^3$, $262^3$, $270^3$, $278^3$, $294^3$, $318^3$, $334^3$, $342^3$, $366^3$, $390^3$, $414^3$, $422^3$, $438^3$, $462^3$, $478^3$, $486^3$, $494^3$, $510^3$, $534^3$, $550^3$, $558^3$, $582^3$, $606^3$, $630^3$, $638^3$, $654^3$, $678^3$, $694^3$, $702^3$, $710^3$, $726^3$, $750^3$, $766^3$, $774^3$, $798^3$, $822^3$, $846^3$, $854^3$, $870^3$, $894^3$, $910^3$, $918^3$, $926^3$, $942^3$, $966^3$, $982^3$, $990^3$,\hfill\par
$6^9$\\\hline
33&$6^3$, $30^3$, $54^3$, $78^3$, $102^3$, $126^3$, $150^3$, $174^3$, $198^3$, $222^3$, $246^3$, $270^3$, $294^3$, $318^3$, $342^3$, $366^3$, $390^3$, $414^3$, $438^3$, $462^3$, $486^3$, $510^3$, $534^3$, $558^3$, $582^3$, $606^3$, $630^3$, $654^3$, $678^3$, $702^3$, $726^3$, $750^3$, $774^3$, $798^3$, $822^3$, $846^3$, $870^3$, $894^3$, $918^3$, $942^3$, $966^3$, $990^3$\\\hline
35&$6^5$, $14^5$, $22^5$, $30^5$, $38^5$, $46^5$, $54^5$, $62^5$,\hfill\par
$14^7$\\\hline
39&$6^3$, $30^3$, $54^3$, $78^3$, $102^3$, $126^3$, $150^3$, $174^3$, $198^3$, $222^3$, $246^3$, $270^3$, $294^3$, $318^3$, $342^3$, $366^3$, $390^3$, $414^3$, $438^3$, $462^3$, $486^3$, $510^3$, $534^3$, $558^3$, $582^3$, $606^3$, $630^3$, $654^3$, $678^3$, $702^3$, $726^3$, $750^3$, $774^3$, $798^3$, $822^3$, $846^3$, $870^3$, $894^3$, $918^3$, $942^3$, $966^3$, $990^3$\\\hline
45&$6^3$, $30^3$, $54^3$, $78^3$, $102^3$, $126^3$, $150^3$, $174^3$, $198^3$, $222^3$, $246^3$, $270^3$, $294^3$, $318^3$, $342^3$, $366^3$, $390^3$, $414^3$, $438^3$, $462^3$, $486^3$, $510^3$, $534^3$, $558^3$, $582^3$, $606^3$, $630^3$, $654^3$, $678^3$, $702^3$, $726^3$, $750^3$, $774^3$, $798^3$, $822^3$, $846^3$, $870^3$, $894^3$, $918^3$, $942^3$, $966^3$, $990^3$,\hfill\par
$6^5$, $14^5$, $22^5$, $30^5$, $38^5$, $46^5$, $54^5$, $62^5$,\hfill\par
$6^9$\\\hline
49&$14^7$

\end{tabular}
\end{table}
\renewcommand{\arraystretch}{1}


\bibliographystyle{alpha}
\bibliography{bibtex}

\end{document}